\documentclass[11pt]{article}
\usepackage{amsfonts}
\usepackage{amsmath}
\usepackage{amssymb}
\usepackage{amstext}
\usepackage{geometry}
\usepackage{graphicx}

\newtheorem{theorem}{Theorem}

\newtheorem{definition}{Definition}

\newtheorem{condition}{Condition}

\newtheorem{remark}{Remark}

\newcommand{\eps}{\varepsilon}

\newcommand{\bd}{\begin{description}}
\newcommand{\ed}{\end{description}}

\newcommand{\R}{{\mathbb R}}

\newcommand{\EE}{{\cal E}}

\def\squarebox#1{\hbox to #1{\hfill\vbox to #1{\vfill}}}

\newcommand{\beaa}{\begin{eqnarray*}}
\newcommand{\eeaa}{\end{eqnarray*}}
\newcommand{\bt}{\begin{theorem}}
\newcommand{\et}{\end{theorem}}

\newcounter{bean}
\newcommand{\benuma}{\setlength{\labelwidth}{.25in}
\begin{list}%
{(\alph{bean})}{\usecounter{bean}}}
\newcommand{\eenuma}{\end{list}}

\newcommand{\be}{\begin{equation}}
\newcommand{\ee}{\end{equation}}

\newcommand{\Df}{\doteq}

\newcommand{\beq}{\begin{eqnarray*}}
\newcommand{\eeq}{\end{eqnarray*}}
\newcommand{\beqn}{\begin{eqnarray}}
\newcommand{\eeqn}{\end{eqnarray}}
\newcommand{\inn}[2]{\langle {#1}, {#2}\rangle}

\newcommand{\ep}{\varepsilon}

\begin{document}

\author{Vasileios Maroulas\thanks{%
Institute for Mathematics and its Applications, University of
Minnesota, Minneapolis, MN 55455, USA. Research supported in part by
IMA. Email:maroulas@ima.umn.edu}.}
\title{Uniform Large Deviations for $\infty-$dimensional stochastic systems with jumps.}
\date{\today}
\maketitle

\begin{abstract}
Uniform large deviation principles for positive functionals of all equivalent types of infinite dimensional Brownian motions acting together with a Poisson random measure are established. The core of our approach is a variational representation formula which for an infinite sequence of i.i.d real Brownian motions and a Poisson random measure was shown in \cite{BDM3}.
\end{abstract}

%

\section{Introduction}
The theory of large deviations is one of the most active research fields in probability, having many applications to areas
such as statistical inference, queueing systems, communication networks, information
theory, risk sensitive control, partial differential equations and statistical mechanics. We
refer the reader to \cite{DeZe,DeSt,DuEl,Var} for background, motivation, applications and fundamental
results in the area. In this paper we establish a general uniform large deviation for functionals of a Poisson random measure (PRM) and infinite dimensional Brownian motion. These two types of driving noises are used in a wide range of processes describing various physical and/or financial phenomena, e.g. reaction-diffusion of particles, environmental pollution, stock return, etc. The uniform large deviation result is expected to be fruitful in the study of asymptotics of steady state behavior for such infinite dimensional stochastic partial differential equations with jumps describing the aforementioned phenomena. The uniformity is with respect to a parameter $\zeta$ which takes values in some compact subset of a Polish space $\mathcal{E}_0$. Typically, $\zeta$ is the initial condition of the corresponding stochastic partial differential equation (SPDE) whose solution's large deviation estimates are considered. A similar  large deviation result for functionals of an infinite dimensional Brownian motion was established in \cite{buddup} and its uniform analogue in \cite{BDM1}. These results were used to study small noise asymptotics for a variety of infinite dimensional stochastic dynamical models and a partial list of such studies is \cite{BeMi,BDM1,BDM2, ChMi, DuDuGa, DuMi, Liu,MaSrSu,ReZh1, ReZh2, RoZhZh, SrSu, WaDu, YaHo, zha3, zha4}) 

Our approach to the large deviation analysis is based on a variational representation for Polish space valued functionals of a PRM and infinite dimensional Brownian motion. Such a variational result was established prior in \cite{BDM3} for an infinite sequence of standard real Brownian motions and a Poisson random measure. Depending on the application, the infinite nature of the Brownain noise may be equivalently expressed as a Brownian sheet, a Hilbert-space valued Brownian motion, or a cylindrical Brownian motion. In this paper the variational representation result for functionals of any tantamount type of infinite dimensional Brownian motion and a PRM will be presented.

A key ingredient in formulating the variational formulation is the appropriate version of controlled PRM and infinite Brownian motion which will be used for purposes of representation. In the Brownian case, the control shifts the mean. In the Poisson random measure case, the control process enters as a
censoring/thinning function, which in turn allows
for elementary weak convergence arguments in proofs of large deviation
results. In \cite{zha}, Zhang has also proved a variational representation for
functionals of a PRM. The corresponding control there moves the atoms of the
Poisson random measure through a rather complex nonlinear transformation. However, the fact that atoms are neither created nor destroyed is partly responsible
for the fact that the representation does not cover the standard Poisson
process. 

The usefulness of the representations is the fact that this approach does not require any exponential probability estimates to be established. 
Exponential continuity (in probability) and exponential tightness estimates are perhaps the hardest and most technical parts of the usual proofs
based on discretization and approximation arguments and this becomes particularly hard in infinite dimensional settings
where these estimates are needed with metrics on exotic function spaces. 
Furthermore what is required for the weak convergence approach, beyond the variational
representations, is that basic qualitative properties (existence, uniqueness
and law of large number limits) can be demonstrated for certain controlled
versions of the original process. 


We now give an outline of the paper.  Section \ref{prelim} contains
some background material on large deviations, infinite
dimensional Brownian motions and a Poisson random measure. In Section \ref{varrep}, we present a variational representation for bounded nonnegative functionals of an infinite sequence of real Brownian motions and PRM. This variational representation, originally obtained in \cite{BDM3}, is the starting point of our study. We also provide analogous representations for other formulations of infinite dimensional Brownian motions and Poisson random measure. Section \ref{ldpgen}, the main section of this paper, gives a uniform large deviation result for Polish space valued functionals of infinite dimensional Brownian motions and Poisson random measure. Sufficient conditions for the uniform LDP for each of the formulations of an infinite dimensional Brownian motion mentioned above are provided.

\noindent \textbf{Notation and a topology.} The following notation will be
used. 
The
Borel sigma-field on $\mathbb{S}$ will be denoted as $\mathcal{B}(\mathbb{S}%
) $. 
 Given $\mathbb{S}$-valued random variables $X_{n},X$, we will write $X_{n}\Rightarrow X$ \
 to denote the weak convergence of $P\circ X_{n}^{-1}$ to $P\circ X^{-1}$.
For a real bounded measurable map $h$ on a measurable space $(V,\mathcal{V)}$%
, we denote $\sup_{v\in V}|h(v)|$ by $\left\Vert h\right\Vert _{\infty }$.

For a locally compact Polish space $\mathbb{S}$, we denote by $\mathcal{M}%
_{F}(\mathbb{S})$ the space of all measures $\nu $ on $(\mathbb{S},\mathcal{B%
}(\mathbb{S}))$, satisfying $\nu (K)<\infty $ for every compact $K\subset
\mathbb{S}$. We endow $\mathcal{M}_{F}(\mathbb{S})$ with the weakest
topology such that for every $f\in C_{c}(\mathbb{S})$ the function $\nu
\mapsto \left\langle f,\nu \right\rangle =\int_{\mathbb{S}}f(u)\,\nu
(du),\nu \in \mathcal{M}_{F}(\mathbb{S})$ is a continuous function. This
topology can be metrized such that $\mathcal{M}_{F}(\mathbb{S})$ is a Polish
space. One metric that is convenient for this purpose is the following.
Consider a sequence of open sets $\left\{ O_{j},j\in \mathbb{N}\right\} $
such that $\bar{O}_{j}\subset O_{j+1}$, each $\bar{O}_{j}$ is compact, and $%
\cup _{j=1}^{\infty }O_{j}=\mathbb{S}$ (cf. Theorem 9.5.21 of \cite{Roy}).
Let $\phi _{j}(x)=\left[ 1-d(x,O_{j})\right] \vee 0$, where $d$ denotes the
metric on $\mathbb{S}$. Given any $\mu \in \mathcal{M}_{F}(\mathbb{S})$, let
$\mu ^{j}\in \mathcal{M}_{F}(\mathbb{S})$ be defined by $\left[ d\mu
^{j}/d\mu \right] (x)=\phi _{j}(x)$. Given $\mu ,\nu \in \mathcal{M}_{F}(%
\mathbb{S})$, let
\begin{equation*}
\bar{d}(\mu ,\nu )=\sum_{j=1}^{\infty }2^{-j}\left\Vert \mu ^{j}-\nu
^{j}\right\Vert _{BL},
\end{equation*}%
where $\left\Vert \cdot \right\Vert _{BL}$ denotes the bounded, Lipschitz
norm:

\begin{equation*}
\left\Vert \mu ^{j}-\nu ^{j}\right\Vert _{BL}=\sup \left\{ \int_{\mathbb{S}%
}fd\mu ^{j}-\int_{\mathbb{S}}fd\nu ^{j}:\left\vert f\right\vert _{\infty
}\leq 1,\left\vert f(x)-f(y)\right\vert \leq d(x,y)\text{ for all }x,y\in
\mathbb{S}\right\} .
\end{equation*}

It is straightforward to check that $\bar{d}(\mu ,\nu )$ defines a metric
under which $\mathcal{M}_{F}(\mathbb{S})$ is a Polish space, and that
convergence in this metric is essentially equivalent to weak convergence on
each compact subset of $\mathbb{X}$. Specifically, $\bar{d}(\mu _{n},\mu
)\rightarrow 0$ if and only if for each $j\in \mathbb{N}$, $\mu
_{n}^{j}\rightarrow \mu ^{j}$ in the weak topology as finite nonnegative
measures, i.e., for all $f\in C_{b}(\mathbb{X})$%
\begin{equation*}
\int_{\mathbb{S}}fd\mu _{n}^{j}\rightarrow \int_{\mathbb{S}}fd\mu ^{j}.
\end{equation*}%
Throughout $\mathcal{B}(\mathcal{M}_{F}(\mathbb{S}))$ will denote the Borel
sigma-field on $\mathcal{M}_{F}(\mathbb{S})$, under this topology.

\section{Preliminaries}

\label{prelim}
In this section we recall some basic definitions and
the equivalence between a LDP and Laplace principle for a family of
probability measures on some Polish space. We next recall some
commonly used formulations for an infinite dimensional Brownian
motion, such as an infinite sequence of i.i.d.\ standard real Brownian motions, a Hilbert space valued Brownian motion, a cylindrical Brownian motion, and  a space-time Brownian sheet.  Relationships between these
various formulations are noted as well. At the end of the section the definition of a Poisson random measure is presented.

\subsection{Large Deviation Principle and Laplace Asymptotics.} Let $%
\{X^{\epsilon },\epsilon >0\} \equiv \{X^{\epsilon }\}$ be a family of random variables defined on a
probability space $(\Omega ,\mathcal{F},\mathbb{P})$ and taking values in a
Polish space $\mathcal{E}$. Denote
the metric on $\mathcal{E}$ by $d(x,y)$ and expectation with respect to $%
\mathbb{P}$ by $\mathbb{E}$. The theory of large deviations is
concerned with events $A$ whose probabilities
$\mathbb{P}(X^{\epsilon }\in A)$ converge to zero exponentially fast
as $\epsilon \rightarrow 0$. The exponential decay rate of such
probabilities is typically expressed in terms of a ``rate function''
$I$ mapping $\mathcal{E}$ into $[0,\infty ]$. If a sequence of random variables satisfies the large deviation principle
with some rate function, then the rate function is unique. In many problems one is interested in obtaining
exponential estimates on functions which are more general than
indicator functions of closed or open sets. This leads to the study
of the Laplace principle, which is tantamount to the LDP. The reader should refer to \cite{DuEl} for all the aforementioned definitions and equivalence between the LDP and the Laplace principle.

In view of this equivalence, the rest of this work will be
concerned with the study of the Laplace principle. In fact we will study a
somewhat strengthened notion, namely a \emph{Uniform Laplace Principle}, as
introduced below. The uniformity is critical in certain applications, such
as the study of exit time and invariant measure asymptotics for small noise
Markov processes \cite{FW1}.

Let $\mathcal{E}_{0}$ and $\mathcal{E}$ be Polish spaces. For each $\epsilon
>0$ and $y\in \mathcal{E}_{0}$ let $X^{\epsilon ,y}$ be $\mathcal{E}$ valued random
variables given on the probability space $(\Omega ,\mathcal{F}%
,\mathbb{P})$ .

\begin{definition}
\label{compact levels} A family of rate functions $I_y$ on $\EE$,
parameterized by $y \in \EE_0$, is said to have compact level sets
on
compacts if for all compact subsets $K$ of $\EE_0$ and each $M< \infty$, $%
\Lambda_{M,K} \doteq \cup_{y\in K}\{x \in \EE: I_y(x) \leq M\}$ is a compact
subset of $\EE$.
\end{definition}

\begin{definition}
\label{defn uniform laplace principle} (Uniform Laplace Principle) Let $%
I_{y} $ be a family of rate functions on $\EE$ parameterized by $y$ in $\EE%
_{0}$ and assume that this family has compact level sets on
compacts. The family $\{X^{\epsilon ,y}\}$ is said to satisfy the
Laplace principle on $\EE $ with rate function $I_{y}$, uniformly on
compacts, if for all compact subsets $K$ of $\mathcal{E}_0$ and all
bounded continuous functions $h$ mapping $\EE$ into $\mathbb{R}$,
\begin{equation*}
\lim_{\epsilon \rightarrow 0}\sup_{y\in K}\left\vert \epsilon \log \mathbb{E}%
_{y}\left\{ \exp \left[ -\frac{1}{\epsilon }h(X^{\epsilon ,y})\right]
\right\} +\inf_{x\in \EE}\Big{\{}h(x)+I_{y}(x)\Big{\}}\right\vert =0.
\end{equation*}
\end{definition}

\subsection{Infinite Dimensional Brownian Motions and Poisson random measure.}
This section revisits basic definitions for infinite dimensional Brownian motions and a Poisson random
measure. We first start with a definition of a Poisson random measure.

\begin{definition} \label{PRM}
Let $(K, \mathcal{K}, \mu)$ be some measure space with
$\sigma$-finite measure $\mu$. The Poisson random measure with
intensity measure $\mu$ is a family of random variables $\{N(A),
A\in\mathcal{K}\}$ defined on some probability space $(\Omega,
\mathcal{F},\mathbb{P})$ such that

\begin{enumerate}

\item $\forall \omega \in \Omega, \;N(\cdot,\omega)$ is a measure
on $(K, \mathcal{K})$.

 \item $\forall A \in \mathcal{K}, \;N(A)$ is
a Poisson random variable with rate $\mu(A)$, i.e.
$\mathbb{P}(N(A)=n)=  \frac{e^{-\mu(A)}}{n!} \mu(A)^n$.

\item If $A_1,A_2,\ldots,A_n \in \mathcal{K}$ disjoint, then
$N(A_1), N(A_2), \cdots, N(A_n)$ are mutually independent.
\end{enumerate}

\end{definition}

The rest of this section deals with all the equivalent types of an infinite dimensional nature of the
Brownian motion, for example depending on the application, an infinite sequence of i.i.d. standard (1--dim)
Brownian motions, a Hilbert space valued Brownian motion, a
cylindrical Brownian motion, and a space-time Brownian sheet. The reader should refer to \cite{BDM1} and references therein for an explanation how these infinite dimensional Brownian motions are related to each other.

Let $(\Omega,\mathcal{F},\mathbb{P})$ be a probability space with
an increasing family of right continuous $\mathbb{P}$--complete
sigma fields $\{\mathcal{F}_t\}_{t \geq 0}$. We will refer to
$(\Omega, \mathcal{F}, \mathbb{P}, \{\mathcal{F}_t\})$ as a
filtered probability space. Let $\{\beta_i\}_{i=1}^\infty$ be an
infinite sequence of independent, standard, one dimensional,
$\{\mathcal{F}_t\}$--Brownian motions given on this filtered
probability space. We will frequently consider all our stochastic
processes defined on a finite time interval $[0,T]$, where $T \in
(0, \infty)$ is a fixed arbitrary terminal time. We denote by
$\mathbb{R}^\infty$, the product space of countably infinite
copies of the real line. Then $\beta= \{\beta_i\}_{i=1}^\infty$ is
a random variable with values in the Polish space
$C([0,T]: \mathbb{R}^\infty)$ and represents the
simplest model for an infinite dimensional Brownian motion.

Frequently in applications it is convenient to express the Brownian
noise, analogous to finite dimensional theory, as a Hilbert space
valued stochastic processes. Let $(H,\langle \cdot, \cdot \rangle)$ be a real
separable Hilbert space. Let $Q$ be a bounded, strictly positive,
trace class operator on $H$.
\begin{definition} \label{defn Q-Wiener process}
An $H$--valued stochastic process $\{W(t), t \geq 0 \}$, given on
a filtered probability space $(\Omega,\mathcal{F},\mathbb{P},\{\mathcal{F}_t)\}$ is
called a $Q$--Wiener process with respect to $\{\mathcal{F}_t\}$
if for every non--zero $h \in H$,
\begin{displaymath}
         \bigl \{ \langle Qh, h \rangle^{-\frac{1}{2}} \langle W(t),h \rangle,\;\{\mathcal{F}_t\}\bigr \}_{ t \geq 0}
\end{displaymath}
is a one--dimensional standard Wiener process.
\end{definition}

\begin{remark}
\label{rem1107} Consider the Hilbert space $l_2 \Df \{x \equiv (x_1, x_2,
\cdots) : x_i \in \R\, \mbox{and}\, \sum x_i^2 < \infty\}$ with the inner
product $\inn{x}{y} \Df \sum x_iy_i$. Let $\{\lambda_i\}_{i=1}^{\infty}$ be
a sequence of strictly positive numbers such that $\sum \lambda_i < \infty$.
Then the Hilbert space $\bar l_2 \Df \{x \equiv (x_1, x_2, \cdots) :x_i \in \R%
\, \mbox{and}\, \sum \lambda_i x_i^2 < \infty\}$ with the inner product $%
\inn{x}{y}_1 \Df \sum \lambda_ix_iy_i$ contains $l_2$ and the
embedding map is Hilbert-Schmidt. Furthermore, the infinite sequence
of real Brownian motions $\beta$ takes values in $\bar l_2$ almost
surely and can be regarded as a $\bar l_2$ valued $Q$--Wiener
process with $\inn{Qx}{y}_1 = \sum_{i=1}^{\infty} \lambda_i^2 x_iy_i
$.
\end{remark}

The trace class operator $Q$ may be interpreted that it injects a ``coloring'' to
a white noise, namely an independent sequence of standard Brownian
motions, in a manner such that the resulting process has better
regularity. In some models of interest, such coloring is obtained
indirectly in terms of (state dependent) diffusion coefficients.
It is natural, in such situations to consider the driving noise as
a ``cylindrical Brownian motion" rather than a Hilbert space
valued Brownian motion. Let $(H, \langle \cdot, \cdot \rangle)$ be a Hilbert space valued Brownian motion and denote the norm on $H$ by $||\cdot||$. Fix a
filtered probability space $(\Omega, \mathcal{F}, \mathbb{P}, \{ \mathcal{F}_t\})$
\begin{definition}\label{defn cbm}
        A family $\{B_t(h) \equiv B(t,h): t \geq 0, h \in H\}$ of real random
        variables is said to be an $\{\mathcal{F}_t\}$--cylindrical Brownian motion
        if:
        \begin{enumerate}
            \item For every $h \in H \mbox{with}\; \;
            ||h||=1$,
            $\{B(t,h), \mathcal{F}_t\}_{t \geq 0 }$ is a standard
            Wiener process.
            \item For every $t \geq 0,\; a_1, a_2 \in\mathbb{R}$ and $f_1,f_2 \in H$,
            \begin{displaymath}
            B(t,a_1f_1+a_2f_2)=a_1B(t,f_1)+a_2B(t,f_2) \; \; a.s.
            \end{displaymath}
            \end{enumerate}
            \end{definition}

In many physical dynamical systems with randomness, the Brownian
noise is given as a space--time white noise process, also referred
to as a Brownian sheet. Let fix a bounded open subset $\mathcal{O}
\subseteq \mathbb{R}^d$.
\begin{definition}\label{defn BS}
A Gaussian family of real--valued random variables $\bigl\{
B(t,x),\; (t,x) \in \mathbb{R}_+ \times \mathcal{O} \bigr\}$ on
the above filtered probability space is called a Brownian sheet if
\begin{enumerate}
\item $\mathbb{E} B(t,x)=0, \; \forall (t,x) \in \mathbb{R}_+
\times \mathcal{O}$ \item $B(t,x)-B(s,x)$ is independent of
$\{\mathcal{F}_s\},\; \forall \; 0 \leq s \leq t$ and $x \in
\mathcal{O}$ \item $Cov \bigl( B(t,x), B(s,y)\bigr)= \lambda
(A_{t,x} \cap A_{s,y})$, where $\lambda$ is the Lebesgue measure
on $\mathbb{R}_+ \times \mathcal{O}$ and $A_{t,x} \doteq \bigl\{
(s,y)\in \mathbb{R}_+ \times \mathcal{O} \big| \; 0 \leq s \leq t
\; \; and \;\; y_j \leq x_j\; j=1, \cdots, d\bigl\}$. \item The
map $(t,u)\mapsto B(t,u)$ from $[0, \infty) \times \mathcal{O}$ to
$\mathbb{R}$ is continuous a.s.
\end{enumerate}
\end{definition}

\section{Variational Representations for functionals of Poisson Random Measure and Brownian motions.}
\label{varrep}

In this section we state the representation for functionals of
both a PRM and infinite dimensional Brownian motions.

Fix $T\in
(0,\infty )$. Let $\mathbb{X}$ be a locally compact Polish space and
$\mathbb{X}_{T}=[0,T]\times \mathbb{X}.$ Fix
a measure $\nu \in \mathcal{M}_{F}(\mathbb{X)}$ and let $\nu
_{T}=\lambda _{T}\otimes \nu $,
where $\lambda _{T}$ is the Lebesgue measure on $[0,T]$. Let $\mathbb{M}=%
\mathcal{M}_{F}(\mathbb{X}_{T}\mathbb{)}$ and denote by
$\mathbb{P}$ the
unique probability measure on $(\mathbb{\mathbb{M}},\mathcal{B(\mathbb{%
\mathbb{M}\mathcal{))}}}$ under which the canonical map, $N:\mathbb{\mathbb{M%
}}\rightarrow \mathbb{M},N(m)\doteq m,$ is a Poisson random
measure with intensity measure $\nu _{T}$. The corresponding
expectation operator will be denoted by $\mathbb{E}$.

Furthermore, denote the product space of
countable infinite copies of the real line by $\mathbb{R}^{\infty
}$. Endowed with the topology of coordinate-wise convergence
$\mathbb{R}^{\infty }$ is a Polish
space. Also let write the Polish space $C([0,T]:\mathbb{R}^{\infty })$ as $%
\mathbb{W}$ and consider the product space $\mathbb{V} \doteq
\mathbb{W\times M}$. Abusing the above notation, let
$N:\mathbb{V\rightarrow M}$ be defined by $N(w,m)=m$, and for the
coordinate maps, $\beta =\{\beta _{i}\}_{i=1}^{\infty }$, on
$\mathbb{V}$ let $\beta _{i}(w,m)=w_{i}$, for any $(w,m)\in
\mathbb{V}$.
Define,
\begin{equation}
\mathcal{G}_{t}\doteq \sigma \left\{
N((0,s]\times A),\beta _{i}(s):0\leq s\leq t,A\in \mathcal{B(\mathbb{X)}}%
,i\geq 1\right\} .
\end{equation}
With applications to large deviations in mind,
for $\theta >0$, denote by $\mathbb{P}_{\theta
}$ the unique probability measure on
$(\mathbb{V},\mathcal{B}(\mathbb{V))}$ such that under
$\mathbb{P}_{\theta }$:

\begin{enumerate}
\item $\{\beta _{i}\}_{i=1}^{\infty }$ is an i.i.d. family of
standard Brownian motions.

\item $N$ is a PRM with intensity measure $\theta \nu _{T}$.

\item $\left\{ \beta _{i}(t),t\in \lbrack 0,T]\right\} $, $\left\{
N([0,t]\times A),t\in \lbrack 0,T]\right\} $ are $\mathcal{G}_{t}$%
-martingales for every $i\geq 1$, $A\in \mathcal{B}(\mathbb{X})$.
\end{enumerate}

Let $\mathbb{Y}=\mathbb{X}\times \lbrack 0,\infty )$ and $\mathbb{Y}%
_{T}=[0,T]\times \mathbb{Y}$. Let $\mathbb{\bar{M}}=\mathcal{M}_{F}(\mathbb{Y%
}_{T}\mathbb{)}$ and let $\mathbb{\bar{P}}$ be the unique
probability
measure on $(\mathbb{\bar{M}},\mathcal{B(}\mathbb{\bar{M}}\mathcal{\mathbb{%
\mathcal{))}}}$ such that the canonical map, $\bar{N}:\mathbb{\bar{M}}%
\rightarrow \mathbb{\bar{M}},\bar{N}(m)\doteq m,$ is a Poisson
random measure with intensity measure $\bar{\nu}_{T}=\lambda
_{T}\otimes \nu \otimes \lambda _{\infty },$ where $\lambda
_{\infty }$ is Lebesgue measure
on $[0,\infty )$. The corresponding expectation operator will be denoted by $%
\bar{\mathbb{E}}$.

Analogously, let define $\bar{\mathbb{V}} \doteq \mathbb{W} \times
\bar{\mathbb{M}}$. Furthermore, abusing notation, let $\bar{N}:
\bar{\mathbb{V}} \rightarrow \bar{\mathbb{M}}$ be $N(w,\bar{m}) =
\bar{m}$ and for the coordinate maps on $\bar{\mathbb{V}}$ let be
denoted again as $\beta =\{\beta _{i}\}_{i=1}^{\infty }$. The
control will act through this additional component of the
underlying point space. Let $\mathcal{G}_{t}\doteq \sigma \left\{
\bar{N}((0,s]\times A),\beta _{i}(s):0\leq s\leq t,A\in \mathcal{B(\mathbb{Y)}}%
,i\geq 1\right\}$, and to
facilitate the use of a martingale representation theorem let $\mathcal{\bar{%
F}}_{t}$ denote the completion under $\mathbb{\bar{P}}$. We denote by $%
\mathcal{\bar{P}}$ the predictable $\sigma $-field on $[0,T]\times
\mathbb{\bar{V}}$ with the filtration $\left\{
\mathcal{\bar{F}}_{t}:0\leq t\leq
T\right\} $ on $(\mathbb{\bar{V}},\mathcal{B(}\mathbb{\bar{V}}\mathcal{%
\mathbb{\mathcal{))}}}$. Let $\bar{\mathcal{A}}$ be the class of all $(%
\mathcal{\bar{P}\mathbb{\otimes
}\mathcal{B(}\mathbb{X)})}\backslash \mathcal{B}[0,\infty )$
measurable maps $\varphi :\mathbb{X}_{T}\times
\mathbb{\bar{V}}\rightarrow \lbrack 0,\infty )$.\ For $\varphi \in \bar{%
\mathcal{A}}$, define a counting process $N^{\varphi }$ on
$\mathbb{X}_{T}$ by
\begin{equation}
N^{\varphi }((0,t]\mathbb{\times }U)=\int_{(0,t]\mathbb{\times }%
U}\int_{(0,\infty )}1_{[0,\varphi
(s,x)]}(r)\bar{N}(ds\,dx\,dr),\;t\in \lbrack 0,T],U\in
\mathcal{B(}\mathbb{X)}.  \label{ins903}
\end{equation}%
$N^{\varphi }$ is to be thought of as a controlled random measure, with $%
\varphi $ selecting the intensity for the points at location $x$ and time $s$%
, in a possibly random but nonanticipating way. Obviously
$N^{\theta }$ has the same distribution on $\mathbb{\bar{V}}$ with
respect to $\mathbb{\bar{P}}
$ as $N$ has on $\mathbb{V}$ with respect to $\mathbb{P}_{\theta }$. $%
N^{\theta }$ therefore plays the role of $N$ on $\mathbb{\bar{V}}$. Define $%
\ell :[0,\infty )\rightarrow \lbrack 0,\infty )$ by
\begin{equation*}
\ell (r)=r\log r-r+1,\;r\in \lbrack 0,\infty ).
\end{equation*}%
For any $\varphi \in \bar{\mathcal{A}}$ the quantity
\begin{equation}
L_{T}(\varphi )=\int_{\mathbb{X}_{T}}\ell (\varphi (t,x,\omega
))\,\nu _{T}(dt\,dx)  \label{cost}
\end{equation}%
is well defined as a $[0,\infty ]-$valued random variable.

Consider the $\ell _{2}$ Hilbert space as defined in Remark \ref{rem1107} and denote%
\begin{equation} \label{little l}
\mathcal{P}_{2}=\left\{ \psi =\{\psi _{i}\}_{i=1}^{\infty }:\psi
_{i}\text{ is
}\mathcal{\bar{P}}\backslash \mathcal{B}(\mathbb{R})\text{ measurable and }%
\int_{0}^{T}||\psi (s)||^{2}ds<\infty \text{, a.s.
}\mathbb{\bar{P}}\right\}
\end{equation}%
and set $\mathcal{U}=\mathcal{P}_{2}\mathcal{\times \bar{A}}$. For
$\psi \in \mathcal{P}_{2}$ define $\tilde{L}_{T}(\psi
)=\frac{1}{2}\int_{0}^{T}||\psi
(s)||^{2}ds$ and for $u=(\psi ,\varphi )\in \mathcal{U}$, set $\bar{L}%
_{T}(u)=L_{T}(\varphi )+$ $\tilde{L}_{T}(\psi )$. \ For $\psi \in \mathcal{P}%
_{2}$, let $\beta ^{\psi }=(\beta _{i}^{\psi })$ be defined as
$\beta _{i}^{\psi }(t)=\beta _{i}(t)+\int_{0}^{t}\psi _{i}(s)ds$,
$t\in \lbrack 0,T] $, $i\in \mathbb{N}$. The following variational 
representation theorem was established in \cite{BDM3}.

\begin{theorem} \label{repn203bm}
Let $F\in M_{b}(\mathbb{\mathbb{V}})$. Then for $\theta >0$,
\begin{equation*}
-\log \mathbb{E}_{\theta }\mathbb{(}e^{-F(\beta ,N)})=-\log \bar{\mathbb{E}}%
\mathbb{(}e^{-F(\beta ,N^{\theta })})=\inf_{u=(\psi ,\varphi )\in \mathcal{U}%
}\bar{\mathbb{E}}\left[ \theta \bar{L}_{T}(u)+F(\beta
^{\sqrt{\theta }\psi },N^{\theta \varphi })\right] .
\end{equation*}
\end{theorem}
As mentioned in the Introduction, depending on the application the infinite nature of the Brownian noise may be written in several other equivalent forms. First, let establish an analogous variational representation for a
functional of Hilbert space valued Brownian motion and a Poisson
random measure. Let $(H,\langle {\cdot },{\cdot }\rangle )$ be a
Hilbert space and let $W$ be an $H$ valued $Q-$Wiener process,
where $Q$ is a bounded, strictly positive, trace class operator on
the Hilbert space $H$. Let $H_{0}=Q^{1/2}H$, then $H_{0}$ is a
Hilbert space with the inner product $\left\langle
h,k\right\rangle _{0}\doteq \left\langle
Q^{-1/2}h,Q^{-1/2}h\right\rangle ,\;h,k\in H_{0}$. Also the
embedding map $i:H_{0}\mapsto H$ is a Hilbert--Schmidt operator
and
$ii^{\ast }=Q$. \ Let $||\cdot ||_{0}$ denote the norm in the Hilbert space $%
H_{0}$.

Furthermore, denote the Polish space $C([0,T]:H)$ by
$\mathbb{W}(H)$ and denote by $\mathbb{V}(H)$ the product space
$\mathbb{W}(H) \times \mathbb{M}$, where $\mathbb{M}$ as defined in the begining of the current section. Let
$\mathbb{\bar{V}}(H)=\mathbb{W}(H) \times \mathbb{\bar{M}}$.
Abusing notation, let $N:\mathbb{V}(H) \rightarrow \mathbb{M}$ be
defined by $N(w,m)=m$, for $(w,m)\in \mathbb{V}(H)$. The map
$\bar{N}:\mathbb{\bar{V}}(H) \rightarrow \mathbb{\bar{M}}$ is
defined analogously. Let $W$ be defined on $\mathbb{V}(H)$ as
$W(w,m)=w(t)$. Analogous maps on $\mathbb{\bar{V}}(H)$ are denoted
again as $W$. Define $\mathcal{G}_{t}\doteq \sigma \left\{
N((0,s]\times A), W(s):0\leq s\leq t,A\in
\mathcal{B(\mathbb{X)}}\right\} $. For $\theta >0$, denote by
$\mathbb{P}_{\theta }$ the unique probability measure on
$(\mathbb{V}(H),\mathcal{B}(\mathbb{V}(H)))$ such that under
$\mathbb{P}_{\theta }$:

\begin{enumerate}
\item $W(t)$ is an $H-$valued $Q-$Wiener process.

\item $N$ is a PRM with intensity measure $\theta \nu _{T}$.

\item $\left\{ W(t),t\in \lbrack 0,T]\right\} $, $\left\{
N([0,t]\times A),t\in \lbrack 0,T]\right\} $ are $\mathcal{G}_{t}$%
-martingales for every $A\in \mathcal{B}(\mathbb{X})$.
\end{enumerate}

Define $(\mathbb{\bar{P}},\left\{ \mathcal{\bar{G}}_{t}\right\} )$
on $(\mathbb{\bar{V}}(H),\mathcal{B}(\mathbb{\bar{V}}(H)))$
analogous to $\left( \mathbb{P}_{\theta },\left\{
\mathcal{G}_{t}\right\} \right) $ by replacing $(N,\theta \nu
_{T})$ with $(\bar{N},\bar{\nu}_{T})$. Now, let consider the $\mathbb{\bar{P}}$-completion of the filtration $%
\left\{ \mathcal{\bar{G}}_{t}\right\} $ and denote it by $\left\{ \mathcal{%
\bar{F}}_{t}\right\} $. We denote by $\mathcal{\bar{P}}$ the predictable $%
\sigma-$field on $[0,T]\times \mathbb{\bar{V}}(H)$ with the filtration $%
\left\{ \mathcal{\bar{F}}_{t}:0\leq t\leq T\right\}$ on $(\mathbb{\bar{V}}(H),%
\mathcal{B}(\mathbb{\bar{V}}(H)))$. Let $\bar{%
\mathcal{A}}$ be the class of all $(\mathcal{\bar{P}\mathbb{\otimes }%
\mathcal{B(}\mathbb{X)})}\backslash \mathcal{B}[0,\infty )$ measurable maps $%
\varphi :\mathbb{X}_{T}\times \mathbb{\bar{V}}\rightarrow \lbrack
0,\infty )$. \ For $\varphi \in \bar{\mathcal{A}}$, define
$L_{T}(\varphi )$ and the counting process $N^{\varphi }$ on
$\mathbb{X}_{T}$ as in (\ref{cost}) and (\ref{ins903})
respectively.

Define%
\begin{equation} \label{p2h}
\mathcal{P}_{2} \equiv \mathcal{P}_{2}(H)=\left\{ \psi: \psi
\text{ is
}\mathcal{\bar{P}}\backslash \mathcal{B}(\mathbb{R})\text{ measurable and }%
\int_{0}^{T}||\psi (s)||_0^{2}ds<\infty \text{, a.s.
}\mathbb{\bar{P}}\right\}
\end{equation}%
and set $\mathcal{U}(H)=\mathcal{P}_{2}(H)\mathcal{\times \bar{A}}$. For
$\psi \in \mathcal{P}_{2}$ define $\tilde{L}_T \equiv
\tilde{L}_{T}^H(\psi )=\frac{1}{2}\int_{0}^{T}||\psi
(s)||_0^{2}ds$ and for $u=(\psi ,\varphi )\in \mathcal{U}$, set $\bar{L}%
_{T}(u)=L_{T}(\varphi )+$ $\tilde{L}_{T}(\psi )$. \ For $\psi \in \mathcal{P} 
_{2}$, let $W ^{\psi }$ be defined as $W^{\psi
}(t)=W(t)+\int_{0}^{t} \psi(s)ds$, $t\in \lbrack 0,T] $. The
following representation follows from Theorem \ref{repn203bm} and the Propostion 1 in \cite{BDM1}.

\begin{theorem}
\label{repn203Hbm} Let $F\in M_{b}(\mathbb{V}(H))$. Then for
$\theta >0$,
\begin{equation*}
-\log \mathbb{E}_{\theta }\mathbb{(}e^{-F(W ,N)})=-\log \bar{\mathbb{E}}%
\mathbb{(}e^{-F(W ,N^{\theta })})=\inf_{u=(\psi ,\varphi )\in \mathcal{U}%
}\bar{\mathbb{E}}\left[ \theta \bar{L}_{T}(u)+F(W ^{\sqrt{\theta
}\psi },N^{\theta \varphi })\right] .
\end{equation*}
\end{theorem}

Finally, we provide the representation theorem for a Brownian
sheet acting together with a Poisson random measure.
Let denote the Polish space $C([0,T] \times \mathcal{O}:\mathbb{R})$ by $%
\mathbb{W}_{BS}$ and denote by $\mathbb{V}_{BS}$ the product space
$\mathbb{W}_{BS}\times \mathbb{M}. $ Let
$\mathbb{\bar{V}}_{BS}=\mathbb{W}_{BS}\times \mathbb{\bar{M}}$.
Abusing notation, let $N:\mathbb{V}_{BS}\rightarrow \mathbb{M}$ be defined by $N(w,m)=m$%
, for $(w,m)\in \mathbb{V}_{BS}$. The map
$\bar{N}:\mathbb{\bar{V}_{BS}\rightarrow \bar{M}}$ is defined
analogously. Let $B(w,m) = w(t,x)$ on $\mathbb{V}_{BS}$ and
analogously on $\mathbb{\bar{V}}_{BS}$ is denoted again as
$B(t,x)$. Define $\mathcal{G}_{t}\doteq \sigma \left\{
N((0,s]\times A),B(s,x):0\leq s\leq t, x \in \mathcal{O}, A\in \mathcal{B(\mathbb{X)}}%
,i\geq 1\right\} $. For $\theta >0$, denote by $\mathbb{P}_{\theta
}$ the unique probability measure on
$(\mathbb{V}_{BS},\mathcal{B}(\mathbb{V}_{BS}))$ such that under
$\mathbb{P}_{\theta }$:

\begin{enumerate}
\item $B$ is a Brownian sheet.

\item $N$ is a PRM with intensity measure $\theta \nu _{T}$.

\item $\left\{ B(t,x),t\in \lbrack 0,T]\right\} $, $\left\{
N([0,t]\times A),t\in \lbrack 0,T]\right\} $ are $\mathcal{G}_{t}$%
-martingales for every $i\geq 1$, $A\in \mathcal{B}(\mathbb{X})$.
\end{enumerate}

Define $(\mathbb{\bar{P}},\left\{ \mathcal{\bar{G}}_{t}\right\} )$ on $(%
\mathbb{\bar{V}}_{BS},\mathcal{B}(\mathbb{\bar{V}}_{BS}))$
analogous to $\left(
\mathbb{P}_{\theta },\left\{ \mathcal{G}_{t}\right\} \right)$ by replacing $%
(N,\theta \nu _{T})$ with $(\bar{N},\bar{\nu}_{T})$.
Define the $\mathbb{\bar{P}}-$completion of the filtration $%
\left\{ \mathcal{\bar{G}}_{t}\right\} $ and denote it by $\left\{ \mathcal{%
\bar{F}}_{t}\right\} $. We denote by $\mathcal{\bar{P}}$ the predictable $%
\sigma-$field on $[0,T]\times \mathcal{O} \times \mathbb{\bar{V}}_{BS}$ with the filtration $%
\left\{ \mathcal{\bar{F}}_{t}:0\leq t\leq T\right\} $ on $(\mathbb{\bar{V}}_{BS},%
\mathcal{B(}\mathbb{\bar{V}}_{BS}\mathcal{\mathbb{\mathcal{))}}}$. Let $\bar{%
\mathcal{A}}$ be the class of all $(\mathcal{\bar{P}\mathbb{\otimes }%
\mathcal{B(}\mathbb{X)})}\backslash \mathcal{B}[0,\infty )$ measurable maps $%
\varphi :\mathbb{X}_{T}\times \mathbb{\bar{V}_{BS}}\rightarrow \lbrack 0,\infty )$%
.\ For $\varphi \in \bar{\mathcal{A}}$, define $L_{T}(\varphi )$
and the counting process $N^{\varphi }$ on $\mathbb{X}_{T}$ as in
(\ref{cost}) and (\ref{ins903}) respectively.

Define
\begin{equation} \label{p2bs}
\mathcal{P}_{2} \equiv \mathcal{P}_{2}^{BS}=\left\{ \psi: \psi
\text{ is
}\mathcal{\bar{P}}\backslash \mathcal{B}(\mathbb{R})\text{ measurable and }%
\int_{0}^{T} \int_{\mathcal{O}}\psi^2(s,x)ds dx <\infty \text{,
a.s. }\mathbb{\bar{P}}\right\}
\end{equation}%
and set $\mathcal{U}^{BS}=\mathcal{P}_{2}^{BS}\mathcal{\times \bar{A}}$. For
$\psi \in \mathcal{P}_{2}$ define $\tilde{L}_{T} \equiv
\tilde{L}^{BS}_{T}(\psi )=\frac{1}{2}\int_{0}^{T}
\int_{\mathcal{O}} \psi
(s,r)^{2} dr ds$ and for $u=(\psi ,\varphi )\in \mathcal{U}$, set $\bar{L}%
_{T}(u) \equiv \bar{L}%
_{T}^{BS}(u)=L_{T}(\varphi )+$ $\tilde{L}_{T}(\psi )$. \ For $\psi \in \mathcal{P}%
_{2}$, let $B ^{\psi }$ be defined as $B^{\psi
}(t,x)=B(t,x)+\int_{0}^{t} \int_{\mathcal{O} \cap
(-\infty,x]}\psi^2(s,y) dy ds$, $t\in \lbrack 0,T] $, $i\in
\mathbb{N}$. We finally remark the following representation for a Brownian sheet and a Poisson random measure follows from Theorem \ref{repn203bm}, Proposition 3 in \cite{BDM1} and an application of Girsanov's Theorem. 

\begin{theorem}
\label{repn203BS} Let $F\in M_{b}(\mathbb{V}_{BS})$. Then for $\theta >0$%
,
\begin{equation*}
-\log \mathbb{E}_{\theta }\mathbb{(}e^{-F(B ,N)})=-\log \bar{\mathbb{E}}%
\mathbb{(}e^{-F(B ,N^{\theta })})=\inf_{u=(\psi ,\varphi )\in \mathcal{U}%
}\bar{\mathbb{E}}\left[ \theta \bar{L}_{T}(u)+F(B^{\sqrt{\theta
}\psi },N^{\theta \varphi })\right] .
\end{equation*}
\end{theorem}

\section{Uniform Large Deviations Estimates.} \label{ldpgen}

This is the central section of this paper where the uniform Laplace
principle for functionals of a Poisson random measure and an
infinite dimensional Brownian motion of any type are verified. The uniformity
is with respect to a parameter $\zeta$ (typically an initial
condition), which takes values in some compact subset of a Polish
space $\mathcal{E}_{0}$. 

Let first consider the case of a Hilbert space valued Wiener process
and then use this case to deduce analogous Laplace principle results for
functionals of a cylindrical Brownian motion and a Brownian sheet acting independently together with a Poisson random measure.  Let $(\Omega ,\mathcal{F},\mathbb{P},\{%
\mathcal{F}_{t}\})$, $(H,\left\langle \cdot ,\cdot \right\rangle
)$, $Q$ be as in Section \ref{prelim} and let $W$ be an $H-$valued
Wiener process with trace class covariance $Q$ given on this
filtered probability space. Let $\mathcal{E}$ be a Polish space,
and for each $\epsilon >0$, let $\mathcal{G}^{\epsilon }:\mathcal{E}_{0}\times \mathbb{V}(H)%
\rightarrow \mathcal{E}$ be a measurable map. We next discuss a
set of sufficient conditions for a uniform large deviation
principle to hold for the family
\begin{equation} \label{fam1}
\bigl \{Z^{\epsilon ,\zeta}\doteq \mathcal{G}^{\epsilon
}(\zeta,\sqrt{\epsilon}W, \epsilon N^{\epsilon^{-1}}) \bigr \}
\mbox{ as } \epsilon \rightarrow 0.
\end{equation}
Let $H_{0}$ be as introduced before and define for $N\in
\mathbb{N}$
\begin{equation}
\tilde{S}^{N}(H_{0})\doteq \left\{ u\in
L^{2}([0,T]:H_{0}):\tilde{L}_T(u)\leq N\right\} . \label{sn}
\end{equation}%
$\tilde{S}^{N}(H_{0})$, we will be
endowed with the topology obtained from the metric $d_{1}(x,y)=\sum_{i=1}^{\infty}\frac{1}{2^{i}}\left\vert \int_{0}^{T}\left\langle
x(s)-y(s),e_{i}(s)\right\rangle _{0}ds\right\vert$ and
refer to this as the weak topology on $S^{N}(H_{0})$.

Also, let%
\begin{equation} \label{twostars}
S^{N}=\left\{ g:X_{T}\rightarrow \lbrack 0,\infty ):L_{T}(g)\leq
N\right\} .
\end{equation}%
A function $g\in S^{N}$ can be identified with a measure $\nu
_{T}^{g}\in
\mathbb{M}$, defined by $\nu _{T}^{g}(A)=\int_{A}g(s,x)\,\nu _{T}(dsdx)$, $%
A\in \mathcal{B}(\mathbb{X}_{T})$. Recalling from the Introducton that convergence in
$\mathbb{M}$ is essentially equivalent to weak convergence on
compact subsets, the superlinear growth of
$\ell $ implies that $\left\{ \nu _{T}^{g}:g\in S^{N}\right\} $ is
a compact subset of $\mathbb{M}$. Throughout we consider the
topology on $S^{N}$ obtained through this
identification which makes $S^{N}$ a compact space. We let $\bar{S}^{N}=%
\tilde{S}^{N}(H_0)\times S^{N}$ with the usual product topology. Recall the product space $\mathcal{U} = \mathcal{P}_2(H) \times \bar{\mathcal{A}}$ and let $\mathbb{S}%
=\cup _{N\geq 1}\bar{S}^{N}$ and let $\mathcal{U}^{N}$ be the space of $\bar{%
S}^{N}$-valued controls:%
\begin{equation} \label{cspace}
\mathcal{U}^{N}=\left\{ u=(\psi ,\varphi )\in \mathcal{U}:u(\omega )\in \bar{%
S}^{N}\text{, }\mathbb{\bar{P}}\text{ a.e. }\omega \right\} .
\end{equation}%

\begin{condition}
\label{maincond1}There exists a measurable map $\mathcal{G}^{0}:\mathcal{E}_0 \times \mathbb{V}(H)%
\rightarrow \mathcal{E}$ such that the following hold.

\begin{enumerate}
\item For $N\in \mathbb{N}$ let $(f_{n},g_{n})$, $(f,g)\in
\bar{S}^{N}$ be such that $(\zeta_n,f_{n},g_{n})\rightarrow
(\zeta,f,g)$. Then
\begin{equation*}
\mathcal{G}^{0}\left( \zeta_n,\int_{0}^{\cdot }f_{n}(s)ds,\nu
_{T}^{g_{n}}\right) \rightarrow \mathcal{G}^{0}\left( \zeta,
\int_{0}^{\cdot }f(s)ds,\nu _{T}^{g}\right) .
\end{equation*}

\item For $N\in \mathbb{N}$ let $u_{\epsilon }=(\psi _{\epsilon
},\varphi
_{\epsilon })$, $u=(\psi ,\varphi )\in \mathcal{U}^{N}$ be such that, as $%
\epsilon \rightarrow 0$, $u_{\epsilon }$ converges in distribution
to $u$ and $\{\zeta^\epsilon\} \subset \mathcal{E}_0,\;
\zeta^\epsilon \rightarrow \zeta, \mbox{ as } \epsilon \rightarrow
0$.
Then%
\begin{equation*}
\mathcal{G}^{\epsilon }\left( \zeta^\epsilon, \sqrt{\epsilon
}W(\cdot) +\int_{0}^{\cdot }\psi _{\epsilon }(s)ds,\,\epsilon
N^{\epsilon ^{-1}\varphi _{\epsilon }}\right) \Rightarrow
\mathcal{G}^{0}\left(\zeta, \int_{0}^{\cdot }\psi (s)ds,\nu
_{T}^{\varphi }\right) .
\end{equation*}
\end{enumerate}
\end{condition}

For $\phi \in \mathcal{E}$, define $\mathbb{S}_{\phi }=\left\{
(f,g)\in \mathbb{S}:\phi =\mathcal{G}^{0}(\zeta,\int_{0}^{\cdot
}f(s)ds,\nu _{T}^{g})\right\} $. Let
$I_\zeta:\mathcal{E}\rightarrow [0,\infty ]$ be defined by
\begin{equation} \label{rateH}
I_\zeta(\phi )=\inf_{q=(f,g)\in \mathbb{S}_{\phi }}\left\{
\bar{L}_{T}(q)\right\} \text{.}
\end{equation}%

\begin{theorem}
\label{uni-lap-hil} Let $Z^{\epsilon ,\zeta}$ be defined as in
(\ref{fam1}) and suppose that Condition \ref{maincond1} holds.
 Suppose that for all $f\in \mathcal{E},\;\zeta \mapsto
I_{\zeta}(f)$ is a lower semi-continuous (l.s.c.) map from
$\mathcal{E}_{0}$ to $[0,\infty ]$. Then,
for all $\zeta \in \mathcal{E}_{0},\;f \mapsto I_{\zeta}(f)$ is a rate function on $%
\mathcal{E}$ and the family $\{I_{\zeta}(\cdot ),\;\zeta\in
\mathcal{E}_{0}\}$ of
rate functions has compact level sets on compacts. Furthermore, the family $%
\{Z^{\epsilon ,\zeta}\}$ satisfies the Laplace principle on $%
\mathcal{E}$, with rate function $I_{\zeta}$, uniformly on compact subsets of $%
\mathcal{E}_{0}$.
\end{theorem}
\begin{proof}
In order to show that $I_\zeta$ is a rate function and that has
compact level sets on compacts, it is enough to demonstrate that
for all compact subsets $K$ of $\mathcal{E}_0$ and each $M<
\infty$,
\begin{displaymath}
\Lambda_{M,K} \doteq \cup_{\zeta \in K} \{ \phi \in \mathcal{E}:
I_\zeta(\phi) \leq M \}
\end{displaymath}
is a compact subset of $\mathcal{E}$. To establish this we will
show that $\Lambda_{M,K} = \cap_{n \geq 1}
\Gamma_{M+\frac{1}{n},K}$ is compact, where $\Gamma_{M,K}= \Bigl
\{\mathcal{G}^0 \bigl( \zeta, \int_0^\cdot f(s)ds, \nu_t^g \bigr):
x \in \mathcal{E}_0, (f,g) \in \bar{S}^M \Bigr\}$. There exists $\zeta \in K$ such that $I_\zeta(f)
\leq M$. We can find for each $n \geq 1,\; (f_n,g_n) \in
\mathbb{S}$ such that, for $\phi \in
\Lambda_{M,K}$, $\phi = \mathcal{G}^0 \bigl( \zeta,
\int_0^\cdot f(s)ds, \nu_t^g \bigr)$ and $\tilde{L}_T(f_n) \leq M
+\frac{1}{n}$ and $L_T(g_n) \leq M + \frac{1}{n}$. In particular
$(f_n,g_n) \in \bar{S}^{M+1/n}$, and thus $\phi \in
\bar{S}^{M+1/n}.$ Since $n \geq 1$ arbitrary, we have
$\Lambda_{M,K} \subseteq \cap_{n \geq 1} \Gamma_{M+1/n,K}$.
Conversely, suppose $\phi \in \bar{S}^{M+1/n}$, such that $\phi =
\mathcal{G}^0 \bigl( \zeta_n, \int_0^\cdot f_n(s)ds, \nu_t^g
\bigr)$. In particular, we have $I_{\zeta_n}(\phi) \leq M+
\frac{1}{n}$. The map $\zeta \mapsto I_\zeta(\phi)$ is lower
semi-continuous and $K$ is compact, and thus sending $n
\rightarrow \infty$ for some $\zeta \in K, \; I_\zeta(\phi) \leq
M$. Thus $\phi \in \Lambda_{M,K}$, and in turn,
$\cap_{n \geq 1} \Gamma_{M+1/n,K} \subseteq \Lambda_{M,K}$
follows. This proves the first part. For the second part of the
theorem consider $\zeta \in \mathcal{E}_0$ and let
$\{\zeta^\epsilon, \epsilon>0 \} \subseteq \mathcal{E}_0$ such
that $\zeta^\epsilon \rightarrow \zeta$, as $\epsilon \rightarrow
0$. Fix a bounded and continuous $F: \mathcal{E} \rightarrow
\mathbb{R}$. It suffices to show the Laplace Principle's upper and
lower bounds, \cite[Section 1.2]{DuEl}, in terms of $I_\zeta$ for the family $Z^{\epsilon, \zeta^\epsilon}$. For notation convenience we will write
$\tilde{S}^N(H_0)$, defined in (\ref{sn}), as $\tilde{S}^N$
and the reader should recall to $\mathcal{P}_2, \; S^N,
\mathcal{U}^N$ as in (\ref{p2h}), (\ref{twostars}) and
(\ref{cspace}) respectively.

\textit{Lower bound:} From Theorem \ref{repn203Hbm}, we have
\begin{equation} \label{lowerp}
-\epsilon \log \bar{\mathbb{E}} \bigl( \exp(-\frac{1}{\epsilon}
F(Z^{\epsilon,\zeta^\epsilon}))\bigr) = \inf_{u=(\psi, \phi) \in
\mathcal{U}} \mathbb{E}[\bar{L}_T(u)+ F \circ
\mathcal{G}^\epsilon(\zeta^\epsilon, \sqrt{\epsilon}W+\int_0^\cdot
\psi(s)ds, \epsilon N^{\epsilon^{-1}\phi})],
\end{equation}
since
$Z^{\epsilon,\zeta^\epsilon}=\mathcal{G}^\epsilon(\zeta^\epsilon,
\sqrt{\epsilon}W, \epsilon N^{\epsilon^{-1}})$ and
$N^{\epsilon^{-1}}$ is a Poisson random measure with intensity
$\epsilon^{-1} \nu_T$. Fix $\delta \in (0,1)$. Then for every
$\epsilon>0$ there exist $u_{\epsilon }=(\psi _{\epsilon },\varphi
_{\epsilon })\in \mathcal{U}$ such that the right hand side of
(\ref{lowerp})
is bounded below by%
\begin{equation}
\bar{\mathbb{E}}\left[ \bar{L}_{T}(u_{\epsilon })+F\circ \mathcal{G}%
^{\epsilon }\left(\zeta^\epsilon, \sqrt{\epsilon }W +\int_{0}^{\cdot }\psi _{\epsilon
}(s)ds,\epsilon N^{\epsilon ^{-1}\varphi _{\epsilon }}\right) \right]
-\delta .  \label{ins1138}
\end{equation}%
Clearly $\bar{\mathbb{E}}(\bar{L}_{T}(u_{\epsilon }))\leq 2||F||_{\infty }+1$.
For $t \in [0,T]$ let,
\begin{equation*}
L_{t}(u_{\epsilon })=\int_{[0,t]}\left( ||\psi _{\epsilon }(s)||^{2}+\int_{%
\mathbb{X}}\ell (\varphi _{\epsilon }(s,x))\,\nu (dx)\right) ds
\end{equation*}%
and define the following sequence of stopping times
\begin{equation*}
\tau _{M}^{\epsilon }=\inf \left\{ t\in \lbrack 0,T]:\bar{L}_{t}(u_{\epsilon
})\geq M\right\} \wedge T.
\end{equation*}%
Now for the pair of processes $u_{\epsilon ,M}=(\psi _{\epsilon
,M},\varphi _{\epsilon ,M})\in \mathcal{U}^{M}$, where
\begin{equation*}
\varphi _{\epsilon ,M}(t,x)=1+[\varphi _{\epsilon }(t,x)-1]1_{[0,\tau
_{M}^{\epsilon }]}(t),\; \; \psi _{\epsilon ,M}(t)=\psi _{\epsilon }(t)1_{[0,\tau
_{M}^{\epsilon }]}(t),\;t\in \lbrack 0,T],\;x\in \mathbb{X}\text{.}
\end{equation*}%
note that%
\begin{equation*}
\mathbb{\bar{P}}(u_{\epsilon }\neq u_{\epsilon ,M})\leq \mathbb{\bar{P}}(%
\bar{L}_{T}(u_{\epsilon })\geq M)\leq \frac{2||F||_{\infty }+1}{M}.
\end{equation*}%
Choose $M$ large enough so that the right side above is bounded by $\delta
/(2||F||_{\infty })$. Thus (\ref{ins1138}) is bounded below
by%
\begin{equation*}
\bar{\mathbb{E}}\left[ \bar{L}_{T}(u_{\epsilon ,M})+F\circ \mathcal{G}%
^{\epsilon }\left( \zeta^\epsilon, \sqrt{\epsilon }W +\int_{0}^{\cdot }\psi _{\epsilon
,M}(s)ds,\epsilon N^{\epsilon ^{-1}\varphi _{\epsilon ,M}}\right) \right]
-2\delta \text{.}
\end{equation*}%
Note that $\left\{ u_{\epsilon ,M}\right\} _{\epsilon >0}$ is a family of $%
\bar{S}^{M}$-valued random variables. Recalling that $\bar{S}^{M}$ is
compact, choose a weakly convergent subsequence and denote by $u=(\psi
,\varphi )$ the weak limit point. From part 2 of Condition \ref{maincond1} we
have that along this subsequence $\mathcal{G}^{\epsilon }(\zeta^\epsilon, \sqrt{\epsilon }%
W +\int_{0}^{\cdot }\psi _{\epsilon ,M}(s)ds,\epsilon N^{\epsilon
^{-1}\varphi _{\epsilon ,M}})$ converges weakly to $\mathcal{G}%
^{0}(\zeta, \int_{0}^{\cdot }\psi (s)ds,\nu _{T}^{\varphi })$. Thus, using Fatou's
lemma and lower semicontinuity properties of the relative entropy function%
\begin{eqnarray*}
\lefteqn{\liminf_{\epsilon \rightarrow 0}-\epsilon \log \mathbb{\bar{E}}%
\left[ \exp{- \frac{1}{\epsilon} F(Z^{\epsilon,\zeta^\epsilon })}\right] } \\
&\geq &\liminf_{\epsilon \rightarrow 0}\bar{\mathbb{E}}\left[ \bar{L}%
_{T}(u_{\epsilon })+F\circ \mathcal{G}^{\epsilon }\left( \zeta^\epsilon, \sqrt{\epsilon }%
W +\int_{0}^{\cdot }\psi _{\epsilon ,M}(s)ds,\epsilon N^{\epsilon
^{-1}\varphi _{\epsilon ,M}}\right) \right] -2\delta \\
&\geq &\bar{\mathbb{E}}\left[ \bar{L}_{T}(u)+F\circ \mathcal{G}^{0}\left( \zeta,
\int_{0}^{\cdot }\psi (s)ds,\nu _{T}^{\varphi }\right) \right] -2\delta \\
&\geq &\inf_{\phi \in \mathbb{U}}\inf_{q\in \mathbb{S}_{\phi }}\left( \bar{L}%
_{T}(q)+F(\phi )\right) -2\delta \\
&=&\inf_{\phi \in \mathbb{U}}(I_\zeta(\phi )+F(\phi ))-2\delta \text{.}
\end{eqnarray*}%
Since $\delta \in (0,1)$ is arbitrary, this completes the proof of the lower
bound.

\emph{Upper Bound}. We need to establish that
\begin{displaymath}
-\epsilon \log \bar{\mathbb{E}} \bigl( \exp(-\frac{1}{\epsilon}
F(Z^{\epsilon,\zeta^\epsilon}))\bigr)  \leq \inf_{\phi \in
\mathbb{U}}(I_\zeta(\phi )+F(\phi ))
\end{displaymath}
Let $\delta \in
(0,1)$ be arbitrary and $\phi _{0}\in \mathbb{U}$ such that%
\begin{equation*}
I_\zeta(\phi _{0})+F(\phi _{0})\leq \inf_{\phi \in \mathbb{U}}(I_\zeta(\phi )+F(\phi
))+\delta \text{.}
\end{equation*}%
Choose $q=(f,g)\in \mathbb{S}_{\phi _{0}}$ such that $\bar{L}_{T}(q)\leq
I_\zeta(\phi _{0})+\delta $ and $\phi _{0}=\mathcal{G}^{0}\left(\zeta, \int_{0}^{\cdot }f(s)ds,\nu _{T}^{g}\right)$.

But according to (\ref{lowerp}) we have that%
\begin{eqnarray*}
\limsup_{\epsilon \rightarrow 0}-\epsilon \log \mathbb{\bar{E}}\left[
e^{-\frac{1}{\epsilon}F(Z^{\epsilon,\zeta^\epsilon})}\right] &\leq &\bar{L}_{T}(q)+\limsup_{%
\epsilon \rightarrow 0}\bar{\mathbb{E}}\left[ F\circ \mathcal{G}^{\epsilon
}\left( \zeta^\epsilon, \sqrt{\epsilon }W +\int_{0}^{\cdot }f(s)ds,\epsilon N^{\epsilon
^{-1}g}\right) \right] \\
&\leq &I_\zeta(\phi _{0})+\delta +F\circ \mathcal{G}^{0}\left(\zeta, \int_{0}^{\cdot
}f(s)ds,\nu _{T}^{g}\right) \\
&=&I_\zeta(\phi _{0})+F(\phi _{0})+\delta \\
&\leq &\inf_{\phi \in \mathbb{U}}(I_\zeta(\phi )+F(\phi ))+2\delta \text{,}
\end{eqnarray*}%
Since $\delta \in (0,1)$ is arbitrary the proof of the theorem is complete.
 $\blacksquare$
\end{proof}

Next let $\beta \equiv \{\beta _{i}\}$ be a sequence of
independent standard
real Brownian motions on $(\Omega ,\mathcal{F},\mathbb{P},\{\mathcal{F}%
_{t}\})$. Recall that $\beta $ is a $(C([0,T]:{\mathbb{R}}^{\infty
}),\mathcal{B}(C([0,T]:{\mathbb{R}}^{\infty })))\equiv (S,\mathcal{%
S})$ valued random variable. For each $\varepsilon >0$ let $\mathcal{G}%
^{\varepsilon }:\mathcal{E}_{0}\times S\rightarrow \mathbb{V}$ be
a measurable map and define
\begin{equation}
Z^{\varepsilon ,\zeta}\doteq \mathcal{G}^{\varepsilon
}(\zeta,\sqrt{\epsilon }\beta , \epsilon N^{\epsilon^{-1}}).
\label{918}
\end{equation}%
We now consider the Laplace principle for the family
$\{Z^{\varepsilon ,\zeta}\}$, as in \ref{918}, and introduce the analog of Condition
\ref{maincond1} for this setting. Define $\tilde{S}^{N}(l_{2})$ as
in \eqref{sn}, with $H_{0}$ there replaced by the Hilbert space
$l_{2}$. The reader should recall $\mathcal{U}$ as in terms of
(\ref{twostars}) and consider $\bar{S}^N \equiv \bar{S}^N(l_2)=
\tilde{S}^{N}(l_{2})
\times S^N$ with the usual product topology. Let $\mathbb{S}%
=\cup _{N\geq 1}\bar{S}^{N}$ and let $\mathcal{U}^{N}$ as defined
in (\ref{cspace}).

\begin{condition}
\label{maincond2}There exists a measurable map
$\mathcal{G}^{0}:\mathcal{E}_0 \times \mathbb{V} \rightarrow
\mathcal{E}$ such that the following hold.
\begin{enumerate}
\item For $N\in \mathbb{N}$ let $(f_{n},g_{n})$, $(f,g)\in
\bar{S}^{N}$ be such that $(\zeta_n,f_{n},g_{n})\rightarrow
(\zeta,f,g)$. Then
\begin{equation*}
\mathcal{G}^{0}\left( \zeta_n,\int_{0}^{\cdot }f_{n}(s)ds,\nu
_{T}^{g_{n}}\right) \rightarrow \mathcal{G}^{0}\left( \zeta,
\int_{0}^{\cdot }f(s)ds,\nu _{T}^{g}\right) .
\end{equation*}

\item For $N\in \mathbb{N}$ let $u_{\epsilon }=(\psi _{\epsilon
},\varphi
_{\epsilon })$, $u=(\psi ,\varphi )\in \mathcal{U}^{N}$ be such that, as $%
\epsilon \rightarrow 0$, $u_{\epsilon }$ converges in distribution
to $u$ and $\{\zeta^\epsilon\} \subset \mathcal{E}_0,\;
\zeta^\epsilon \rightarrow \zeta, \mbox{ as } \epsilon \rightarrow
0$.
Then%
\begin{equation*}
\mathcal{G}^{\epsilon }\left( \zeta^\epsilon, \sqrt{\epsilon }\beta
+\int_{0}^{\cdot }\psi _{\epsilon }(s)ds,\,\epsilon N^{\epsilon
^{-1}\varphi _{\epsilon }}\right) \Rightarrow
\mathcal{G}^{0}\left(\zeta, \int_{0}^{\cdot }\psi (s)ds,\nu
_{T}^{\varphi }\right) .
\end{equation*}
\end{enumerate}
\end{condition}

For $\phi \in \mathcal{E}$, define $\mathbb{S}_{\phi }=\left\{
(f,g)\in \mathbb{S}:\phi =\mathcal{G}^{0}(\zeta,\int_{0}^{\cdot
}f(s)ds,\nu _{T}^{g})\right\} $. Let $I_\zeta:\mathcal{E}
\rightarrow \lbrack 0,\infty ]$ be defined by
\begin{equation}
I_\zeta(\phi )=\inf_{q=(f,g)\in \mathbb{S}_{\phi }}\left\{
\bar{L}_{T}(q)\right\} \text{.}  \label{insrate1}
\end{equation}%

\begin{theorem}
\label{uni-lap-hil2} Let $Z^{\epsilon ,\zeta}$ be defined as in
(\ref{918}) and suppose that Condition \ref{maincond2} holds.
Suppose that for all $f\in \mathcal{E},\;\zeta \mapsto
I_{\zeta}(f)$ is a lower semi-continuous (l.s.c.) map from
$\mathcal{E}_{0}$ to $[0,\infty ]$, where $I_{\zeta}$ as in (\ref{insrate1}). Then,
for all $\zeta \in \mathcal{E}_{0},\;f \mapsto I_{\zeta}(f)$ is a rate function on $%
\mathcal{E}$ and the family $\{I_{\zeta}(\cdot ),\;\zeta\in
\mathcal{E}_{0}\}$ of
rate functions has compact level sets on compacts. Furthermore, the family $%
\{Z^{\epsilon ,\zeta}\}$ satisfies the Laplace principle on $%
\mathcal{E}$, with rate function $I_{\zeta}$, uniformly on compact
subsets of $\mathcal{E}_{0}$.
\end{theorem}

\begin{proof} From Remark
\ref{rem1107} we can regard $\beta $ as an $H$ valued $Q$--Wiener
process, where $H=\bar l_{2}$ and $Q$ is a trace class operator, as
defined in Remark \ref{rem1107}. Also,
one can check that $H_{0}\doteq Q^{1/2}H=l_{2}$. Since the embedding map $%
i:C([0,T]:\bar l_{2})\rightarrow C([0,T]:{\mathbb{R}}^{\infty })$ is
continuous, $\hat{\mathcal{G}}^{\varepsilon }:\mathcal{E%
}_{0}\times \mathbb{V}(\bar l_{2})\rightarrow \mathcal{E}$ defined as $\hat{%
\mathcal{G}}^{\varepsilon }(\zeta,\sqrt{\ep}v, \epsilon N^{\epsilon^{-1}})
\doteq \mathcal{G}^{\varepsilon }(\zeta,\sqrt{\ep}i(v),\epsilon N^{\epsilon^{-1}})$, $%
(\zeta,v)\in \mathcal{E}_{0}\times C([0,T]:\bar l_{2})$ is a measurable map
for every $\varepsilon \geq 0$. Note also that for $\varepsilon >0$, $%
Z^{\varepsilon ,\zeta^\epsilon}=\hat{\mathcal{G}}^{\varepsilon
}(\zeta,\sqrt{\ep}\beta, \ep N^{\ep^{-1}} )$ a.s. Since Condition
\ref{maincond2} holds, we have that both parts of Condition
\ref{maincond1} are satisfied with $\mathcal{G}^{\varepsilon }$
there
replaced by $\hat{\mathcal{G}}^{\varepsilon }$ for $\varepsilon \geq 0$ and $%
W$ replaced with $\beta $. Define $\hat{I}_{\zeta}(\phi)$ by the right side of %
\eqref{rateH} with $\mathcal{G}^{0}$ replaced by $\hat{\mathcal{G}}%
^{0}$. Clearly $I_{\zeta}(\phi)=\hat{I}_{\zeta}(\phi)$ for all $(x,f)\in \mathcal{E}%
_{0}\times \mathcal{E}$. The result is now an immediate consequence of
Theorem \ref{uni-lap-hil}. $\blacksquare$
\end{proof}

Finally, we consider the uniform Laplace principle for functionals
of a Brownian sheet and a Poisson random measure. Let $B$ be a
Brownian sheet as in
Definition \ref{defn BS}. Let $\mathcal{G}^{\varepsilon }:\mathcal{E}%
_{0}\times \mathbb{V}_{BS}\rightarrow \mathcal{E}$, $\varepsilon
>0$ be a family of measurable maps. Define
\begin{equation} \label{fam3}
Z^{\varepsilon ,\zeta}\doteq \mathcal{G}^{\varepsilon
}(\zeta,\sqrt{\epsilon }B,\epsilon N^{\epsilon^{-1}}).
\end{equation}

We now provide
sufficient conditions for Laplace principle to hold for the family $%
\{Z^{\varepsilon ,\zeta}\}$.

Analogous to classes defined in \eqref{sn}, we introduce
\[
\tilde{S}^{N}\doteq \left\{ \phi \in L^{2}([0,T]\times \mathcal{O}%
):\int_{[0,T]\times \mathcal{O}}\phi ^{2}(s,r)dsdr\leq N\right\} .
\]
Once more, $\tilde{S}^{N}$ is endowed with the weak topology on
$L^{2}([0,T]\times \mathcal{O})$, under which it is a compact
metric space. For $u\in L^{2}([0,T]\times \mathcal{O})$, define
Int$(u)\in C([0,T]\times \mathcal{O}:\mathbb{R})$ by
\begin{equation}
\text{Int}(u)(t,x)\doteq \int_{\lbrack 0,t]\times (\mathcal{O}\cap (-\infty
,x])}u(s,y)dsdy,  \label{definteg}
\end{equation}%
where $(-\infty ,x]=\{y:y_{i}\leq x_{i} \mbox{ for all }
i=1,\cdots ,d\}$. Consider $\bar{S}^N \equiv \bar{S}^N_{BS}=
\tilde{S}^{N} \times S^N$ with the usual product topology. Let $\mathbb{S}%
=\cup _{N\geq 1}\bar{S}^{N}$ and let $\mathcal{U}^{N}$ as defined
in (\ref{cspace}).

\begin{condition}
\label{maincond3}There exists a measurable map $\mathcal{G}^{0}:\mathcal{E}_0 \times \mathbb{V}_{BS}%
\rightarrow \mathbb{U}$ such that the following hold.

\begin{enumerate}
\item For $N\in \mathbb{N}$ let $(f_{n},g_{n})$, $(f,g)\in
\bar{S}^{N}$ be such that $(\zeta_n,f_{n},g_{n})\rightarrow
(\zeta,f,g)$. Then
\begin{equation*}
\mathcal{G}^{0}\left( \zeta_n,Int(f_{n}),\nu _{T}^{g_{n}}\right) \rightarrow
\mathcal{G}^{0}\left( \zeta , Int(f(s)),\nu _{T}^{g}\right) .
\end{equation*}

\item For $N\in \mathbb{N}$ let $u_{\epsilon }=(\psi _{\epsilon
},\varphi
_{\epsilon })$, $u=(\psi ,\varphi )\in \mathcal{U}^{N}$ be such that, as $%
\epsilon \rightarrow 0$, $u_{\epsilon }$ converges in distribution
to $u$ and $\{\zeta^\epsilon\} \subset \mathcal{E}_0,\;
\zeta^\epsilon \rightarrow \zeta, \mbox{ as } \epsilon \rightarrow
0$.
Then%
\begin{equation*}
\mathcal{G}^{\epsilon }\left( \zeta^\epsilon, \sqrt{\epsilon }B
+Int(\psi _{\epsilon
}),\,\epsilon N^{\epsilon ^{-1}\varphi _{\epsilon }}\right)
\Rightarrow \mathcal{G}^{0}\left(\zeta, Int(\psi),\nu _{T}^{\varphi
}\right) .
\end{equation*}
\end{enumerate}

\end{condition}

For $\phi \in \mathcal{E}$, define $\mathbb{S}_{\phi }=\left\{
(f,g)\in \mathbb{S}:\phi =\mathcal{G}^{0}(\zeta,\int_{[0,t]\times
(\mathcal{O} \cap (-\infty,x])}f(s)ds,\nu _{T}^{g})\right\} $. Let
$I_\zeta:\mathcal{E\rightarrow \lbrack }0,\infty ]$ be defined by
\begin{equation}
I_\zeta(\phi )=\inf_{q=(f,g)\in \mathbb{S}_{\phi }}\left\{
\bar{L}_{T}(q)\right\} \text{.}  \label{insrate2}
\end{equation}%

\begin{theorem}
\label{uni-lap-hil3} Let $Z^{\epsilon ,\zeta}$ be defined as in
(\ref{fam3}) and suppose that Condition \ref{maincond3} holds.
Suppose that for all $f\in \mathcal{E},\;\zeta \mapsto
I_{\zeta}(f)$ is a lower semi-continuous (l.s.c.) map from
$\mathcal{E}_{0}$ to $[0,\infty ]$, where $I_\zeta$ as in (\ref{insrate2}). Then,
for all $\zeta \in \mathcal{E}_{0},\;f \mapsto I_{\zeta}(f)$ is a rate function on $%
\mathcal{E}$ and the family $\{I_{\zeta}(\cdot ),\;\zeta\in
\mathcal{E}_{0}\}$ of
rate functions has compact level sets on compacts. Furthermore, the family $%
\{Z^{\epsilon ,\zeta}\}$ satisfies the Laplace principle on $%
\mathcal{E}$, with rate function $I_{\zeta}$, uniformly on compact subsets of $%
\mathcal{E}_{0}$.
\end{theorem}
\begin{proof}
Let $\{e _{i}\}_{i=1}^{\infty }$ be a complete orthonormal system
in $L^{2}(\mathcal{O})$ and let
\begin{equation*}
\beta _{i}(t)\doteq \int_{\lbrack 0,t]\times \mathcal{O}} e_{i}(x)B(dsdx),\;\;t\in \lbrack 0,T],
\;i=1,2,\cdots .
\end{equation*}%
Then $\beta \equiv \{\beta _{i}\}$ is a sequence of independent
standard real Brownian motions and can be regarded as an
$(S,\mathcal{S})-$valued random variable. Now, from
\cite[Proposition 3]{BDM1}, there is a measurable map $h:C([0,T]:{\mathbb{R}}%
^{\infty })\rightarrow C([0,T]\times \mathcal{O}:{\mathbb{R}})$ such that $%
h(\beta )=B$ a.s. Define, for $\varepsilon >0$, $\hat{\mathcal{G}}%
^{\varepsilon }:\mathcal{E}_{0}\times \mathbb{V} \rightarrow
\mathcal{E}$ as $\hat{\mathcal{G}}^{\varepsilon
}(\zeta,\sqrt{\ep}v,\ep N^{\eps^{-1}})\doteq
\mathcal{G}^{\varepsilon }(\zeta,\sqrt{\ep}h(v), \ep
N^{\eps^{-1}})$, $(\zeta,v,\phi) \in \mathcal{E}_{0}\times
\tilde{S}^N(l_2)$. Clearly $\hat{\mathcal{G}}^{\varepsilon }$ is a
measurable map and 
\begin{displaymath}
\hat{\mathcal{G}}^{\varepsilon
}(\zeta,\sqrt{\ep}\beta, \ep N^{\eps^{-1}} )=Z^{\varepsilon
,\zeta} \mbox{ a.s.} 
\end{displaymath}
Next, note that
\begin{equation*}
S_{ac}\doteq \left\{ v\in C([0,T]:{\mathbb{R}}^{\infty }):v(t)=\int_{0}^{t}%
\hat{u}(s)ds,\,t\in \lbrack 0,T],\,\mbox{for some}\,\hat{u}\in
L^{2}([0,T]:l_{2})\right\}
\end{equation*}%
is a measurable subset of $S$. For $\hat{u}\in L^{2}([0,T]:l_{2})$, define $%
u_{\hat{u}}\in L^{2}([0,T]\times \mathcal{O})$ as
\begin{equation*}
u_{\hat{u}}(t,x)=\sum_{i=1}^{\infty }\hat{u}_{i}(t) e _{i}(x),\;(t,x)\in
\lbrack 0,T]\times \mathcal{O}.
\end{equation*}%
Define $\hat{\mathcal{G}}^{0}:\mathcal{E}_{0}\times
\mathbb{V}\rightarrow \mathcal{E}$ as
\begin{equation*}
\hat{\mathcal{G}}^{0}(\zeta,\int_{0}^{\cdot
}\hat{u}(s)ds, \nu_T^g)\doteq
\mathcal{G}^{0}(\zeta,\mbox{Int}(u_{\hat{u}}),\nu_T^g)
\end{equation*}%
and note that
\begin{equation*}
\left\{ \hat{\mathcal{G}}^{0}\left( \zeta,\int_{0}^{\cdot }\hat{u}(s)ds,
\nu_T^g \right) :%
\hat{u}\in S^{M}(l_{2}), \zeta \in K, K \subset \mathcal{E}_0 \right\} =\left\{
\mathcal{G}^{0}\left( \zeta, \mbox{Int}(u), \nu_T^g \right) :u\in
S^{M}, \zeta \in K \subset \mathcal{E}_0 \right\} .
\end{equation*}%
Since Condition \ref{maincond1} holds, we have that its first part
holds with $\mathcal{G}^{0}$ there replaced by $\hat{\mathcal{G}}%
^{0}$. Next, an application of Girsanov's theorem gives that, for every $%
\hat{u}^{\varepsilon }\in \tilde{S}^M(l_{2})$
\begin{equation*}
h\left( \beta +\frac{1}{\sqrt{\epsilon }}\int_{0}^{\cdot
}\hat{u}^{\epsilon
}(s)ds\right) =B+\frac{1}{\sqrt{\epsilon }}\mbox{Int}(u_{\hat{u}%
^{\varepsilon }}),\,a.s.
\end{equation*}%
In particular for every $M<\infty $ and families
$\{\hat{u}^{\epsilon },\phi^\epsilon\}\subset \bar{S}^{M}(l_{2})$
and $\{\zeta^{\varepsilon }\}\subset \mathcal{E}_{0}$, such that
$\{\hat{u}^{\epsilon },\phi^\epsilon\}$ converges in distribution
to $\{\hat{u},\phi\}$ and $\zeta^{\varepsilon }\rightarrow \zeta$,
we have, as $\varepsilon \rightarrow 0$,
\begin{eqnarray*}
\hat{\mathcal{G}}^{\epsilon }\left( \zeta^{\varepsilon },\sqrt{%
\epsilon }\beta +\int_{0}^{\cdot }\hat{u}^{\epsilon }(s)ds,
\epsilon N^{\epsilon^{-1}\phi_\epsilon} \right) &=&\mathcal{G}
^{\epsilon }\left( \zeta^{\varepsilon },\sqrt{\epsilon }B+\mbox{Int}%
(u_{\hat{u}^{\epsilon }}), \epsilon N^{\epsilon^{-1}\phi_\epsilon}\right)  \\
&\Rightarrow &\mathcal{G}^{0}\left( \zeta,\mbox{Int}(u_{\hat{u}}), \nu_T^\phi\right)  \\
&=&\hat{\mathcal{G}}^{0}\left( \zeta,\int_{0}^{\cdot
}\hat{u}(s)ds, \nu_T^\phi \right) .
\end{eqnarray*}%
Thus second part of Condition \ref{maincond1} is satisfied with $\mathcal{G}%
^{\varepsilon }$ replaced by $\hat{\mathcal{G}}^{\varepsilon }$, $%
\varepsilon \geq 0$. The result now follows on noting that if
$\hat{I}_{\zeta}(f) $ is defined by the right side of
\eqref{insrate1} on replacing $\mathcal{G}^{0}
$ there by $\hat{\mathcal{G}}^{0}$, then $\hat{I}_{\zeta}(f)=I_{\zeta}(f)$ for all $%
(\zeta,f)\in \mathcal{E}_{0}\times \mathcal{E}$. $\blacksquare$
\end{proof}

\bibliographystyle{plain}

\end{document}